\documentclass[leqno]{amsart}
\usepackage{amssymb, amsmath}
\usepackage{graphics}
\usepackage[dvips]{graphicx}
\usepackage{eepic}
\usepackage{epic}

\usepackage{color}

\newtheorem{thm}{Theorem}[section]

\newtheorem{lmm}[thm]{Lemma}

\newtheorem{prp}[thm]{Proposition}

\theoremstyle{definition}

\theoremstyle{remark}

\renewcommand{\a}{\alpha}
\renewcommand{\b}{\beta}

\newcommand{\de}{\delta}

\newcommand{\ga}{\gamma}

\newcommand{\Om}{\Omega}
\newcommand{\lan}{\langle}
\newcommand{\ran}{\rangle}

\def\A{{\mathcal{A}}}
\def\R{{\mathbb{R}}}
\def\E{{\mathbb{E}}}
\def\N{{\mathbb{N}}}

\def\M{{\mathcal{M}}}

\def\cP{{\mathcal{P}}}
\def\cE{{\mathcal{E}}}

\def\HDE{{\rm (HDE)}_\theta}

\title[Hydrostatic limit for exclusion process]{Hydrostatic limit for exclusion process with slow boundary revisited}
\author{\textsc{Kenkichi Tsunoda}}
\address{Department of Mathematics, Osaka University,
Osaka, 560-0043, Japan.\newline e-mail: \texttt{k-tsunoda@math.sci.osaka-u.ac.jp}.}

\begin{document}

\maketitle

\begin{abstract}

We revisit in this short article the hydrostatic limit for the exclusion process with slow boundary.
The original proof of this result relies on estimates of the correlation functions.
We achieve the same result based on analysis of two different time scales,
which do not need any information about the correlation functions.

\end{abstract}

\section{Introduction}\label{sec1}
We study in this article the limiting behavior of the empirical measure under
the stationary state, called {\it hydrostatic limit}, for the exclusion process with slow boundary.
This model has been introduced in R. Baldasso, O. Menezes, A. Neumann and R. R. Souza \cite{bmns17},
and can be described as follows. Let $N\in\N$ be a scaling parameter.
Each particle in the bulk $\{1,\dots, N-1\}$ behaves as an independent simple random walk with exclusive constraints.
The terminology {\it slow boundary} means that particles are created or annihilated at the boundary,
at a rate proportional to $N^{-\theta}$ for some $\theta\ge0$.
It has been established in Baldasso et al. \cite{bmns17} that
the following phase transition occurs: the boundary condition of the hydrodynamic equation
is governed by Dirichlet boundary if $\theta<1$, Robin boundary if $\theta=1$ and Neumann boundary if $\theta>1$, respectively.
We omit to introduce more detailed description and a historical background of this model here,
see Baldasso et al. \cite{bmns17} and references therein.

The purpose of this article is to introduce another proof of the hydrostatic limit,
stated in Theorem \ref{11}.
It is worth mentioning that our proof does not use any information about the correlation functions,
while the original one strongly relies on estimates of the correlation functions.
Although our proof can be applied to other particle systems, we concentrate on the exclusion process with slow boundary
in this article to make the presentation simplest.

The original method we follow in this article has been introduced in  Farfan, Landim and Mourragui \cite{flm11},
to show the hydrostatic limit for the boundary gradient driven symmetric exclusion process.
This method has been generalized to the case of the reaction-diffusion model in Landim and Tsunoda \cite{lt18}.
In fact, Landim and Tsunoda's method is robust enough to imply Theorem \ref{11} for $\theta\le1$, see Section \ref{sec3}.
However, the result established in Section \ref{sec3} is not enough to deduce Theorem \ref{11} for $\theta>1$.
This issue will be examined in the first paragraph of Section \ref{sec4}.
To complete the proof of Theorem \ref{11}, we further develop Landim and Tsunoda's method in the case $\theta>1$, where
the boundary condition of the hydrodynamic equation is governed by Neumann boundary conditions.
The method developed in Section \ref{sec4}, which is a main contribution of this article,
seems somewhat new and may be of interest in other contexts.

We remark on several papers related to this work, but only on papers published after Baldasso et al. \cite{bmns17}.
The main motivation of this work is based on recent developments on stationary nonequilibrium states.
See Bertini et al \cite{bdgjl15} for this subject. 
The equilibrium and non-equilibrium fluctuations for the exclusion process with slow boundary
and related models have been investigated
in a series of studies by T. Franco, P. Gon\c{c}alves and A. Neumann and their collaborators: \cite{fgn18, efgnt18, gjnm18}.
The large deviation for the exclusion process with a slow bond is examined in T. Franco and A. Neumann \cite{fn17}.

This article is organized as follows. In Section \ref{sec2}, we introduce the exclusion process with slow boundary precisely.
We also examine results on the hydrodynamic and hydrostatic limit, established in Baldasso et al. \cite{bmns17},
in Subsections \ref{2.2}, \ref{2.3}, respectively. The original proof of the hydrostatic limit is also examined in Subsection \ref{2.3}.
In Sections \ref{sec3}, \ref{sec4}, we study the diffusive time scale or a certain sub-diffusive time scale,
and deduce Theorem \ref{11} for $\theta\le1$ and for $\theta>1$, respectively.

\section{Model and main result}\label{sec2}
We introduce in this section the exclusion process with slow boundary
and state the hydrodynamic and hydrostatic limit for this particle system.
We constantly refer the reader to Baldasso et al. \cite{bmns17}
as most of the statements in this section borrow from the ones of Baldasso et al. \cite{bmns17}.

\subsection{Exclusion process with slow boundary}\label{2.1}

For each $N\in\N$, let $I_N$ be the one-dimensional discrete interval $\{1,\dots,N-1\}$.
Elements of $I_N$ are represented by the letters $x,y$ and $z$, while  an element of
the continuum interval $[0,1]$ is represented by the letter $u$.
Denote the {\it configuration space} by $\Om_N=\{0,1\}^{I_N}$,
and its element, called {\it configuration}, by $\eta=\{\eta(x):x\in I_N\}$.
For each $x\in I_N$, $\eta(x)$ represents the number of particles
sitting at site $x$, in other words, $\eta(x)=1$ if there is a particle at site $x$,
$\eta(x)=0$ otherwise.
For a configuration $\eta\in\Om_N$, let $\eta^{x,y}$ and $\eta^x$ be the configurations
obtained from $\eta$ by exchanging the occupation variables $\eta(x)$ and $\eta(y)$, 
by flipping the occupation variable $\eta(x)$, respectively:
\begin{align*}
\eta^{x,y}(z) \;=\; \begin{cases}
     \eta(y)  & \text{if $z=x$}\;, \\
     \eta(x)  & \text{if $z=y$}\;, \\
     \eta(z)    & \text{otherwise}\;,
\end{cases}
\quad
\eta^x(z) \;=\; \begin{cases}
     1-\eta(x)  & \text{if $z = x$}\;, \\
     \eta(z)  & \text{if $z\neq x$}\;. 
\end{cases}
\end{align*}

We introduce the exclusion process with slow boundary, which is a Markov process on $\Om_N$
whose generator is given by
\begin{align*}
L_N\;=\; L_{N,0} + L_{N,b}^\a + L_{N,b}^\b\;,
\end{align*}
with some fixed $\a, \b\in(0,1)$.
In the previous formula, $L_{N,0}$ stands for the generator of the symmetric simple exclusion process in $I_N$,
 that is, $L_{N,0}$ acts on functions $f:\Om_N\to\R$ as
\begin{align*}
L_{N,0}f(\eta)\;=\; \sum_{x=1}^{N-2} \left[ f(\eta^{x,x+1}) -f(\eta) \right]\;.
\end{align*}
On the other hand, $L_{N,b}^\a$ and $L_{N,b}^\b$ correspond to the dynamics at the left and right boundary, respectively,
which act functions $f:\Om_N\to\R$ as
\begin{align*}
L_{N,b}^\a f(\eta)\;&=\;cN^{-\theta}r_\a(\eta)\left[ f(\eta^1)-f(\eta)\right]\;,\\
L_{N,b}^\b f(\eta)\;&=\;cN^{-\theta}r_\b(\eta)\left[ f(\eta^{N-1})-f(\eta)\right]\;,
\end{align*}
where
\begin{align*}
r_\a(\eta)\;&=\; \a\left[ 1-\eta(1) \right] +( 1-\a)\eta(1)\;, \\
r_\b(\eta)\;&=\; \b\left[1-\eta(N-1)\right] +(1-\b)\eta(N-1)\;,
\end{align*}
with some fixed $c>0$ and $\theta\ge0$.

Denote  by $\nu_\rho^N$ the product Bernoulli measure on $\Om_N$ with density $\rho\in[0,1]$.
It is well known that $\nu_\rho^N$ is symmetric with respect to $L_{N,0}$ for any $\rho\in[0,1]$.
Since $r_\a$ and $r_\b$ are chosen to satisfy the detailed balance conditions with respect to $\nu_\a^N$ and $\nu_\b^N$,
these measures are symmetric with respect to $L_{N,b}^\a$ and $L_{N,b}^\b$, respectively.
However, it is also well known that the Bernoulli measures are not invariant with respect to $L_N$ unless $\a=\b$.
Since the cardinality of the state space $\Om_N$ is finite and the Markov process corresponding to $L_N$ is irreducible,
there exists a unique stationary state, denoted by $\mu_N^{ss}$, which is invariant under the dynamics.

\subsection{Hydrodynamic limit}\label{2.2}
It has been investigated in Baldasso et al. \cite{bmns17} that the boundary condition of
the hydrodynamic equation depends on the parameter $\theta$.
More precisely, the hydrodynamic behavior of the exclusion process with slow boundary is described as follows.
Assume for a while that the macroscopic density at time $0$ is given by
a measurable function $\rho_0:[0,1]\to[0,1]$.
For any $\theta\ge0$, the system in the bulk evolves according to the heat equation in $(0,1)$:
\begin{align*}
\begin{cases}
\partial_t\rho(t,u)\;=\; \partial_u^2\rho(t,u)\;,\\
\rho(0,u)\;=\;\rho_0(u)\;,
\end{cases}
\end{align*}
where $\rho(t,u)$ stands for the macroscopic density at time $t\ge0$ and position $u\in[0,1]$.
In the case $\theta<1$, the boundary condition is governed by Dirichlet boundary conditions:
\begin{align*}
\begin{cases}
\rho(t,0)\;=\; \a\;,\\
\rho(t,1)\;=\;\b \;.
\end{cases}
\end{align*}
In the case $\theta=1$, the boundary condition is governed by Robin boundary conditions:
\begin{align*}
\begin{cases}
\partial_u\rho(t,0)\;=\; c\left[\rho(t,0)-\a\right]\;,\\
\partial_u\rho(t,1)\;=\; c\left[\b-\rho(t,1)\right]\;.
\end{cases}
\end{align*}
In the case $\theta>1$, the boundary condition is governed by Neumann boundary conditions:
\begin{align*}
\begin{cases}
\partial_u\rho(t,0)\;=\; 0\;,\\
\partial_u\rho(t,1)\;=\; 0 \;.
\end{cases}
\end{align*}
We do not review precise definitions of weak solutions to these Cauchy problems here,
see \cite[Subsection 2.3]{bmns17} for them.
For each $\theta\ge0$, denote these Cauchy problems by ${\rm (HDE)}_\theta$.

For each $N\in\N$, denote by $\{S^N_t : t \ge 0\}$ the semigroup associated to
the Markov process generated by $N^2L_N$ and by $\mu_N$ a given initial distribution.
Note that the distribution of the process at time $t$ is given by $\mu_NS_t^N$.

The following result has been established in Baldasso et al. \cite{bmns17}.

\begin{thm}[Hydrodynamic limit]
Assume that the initial distribution $\mu_N$ is associated to 
a measurable function $\rho_0:[0,1]\to[0,1]$.
Namely, it holds for any $\delta>0$ and continuous function $H:[0,1]\to\R$ that
\begin{align*}
\lim_{N\to\infty}\mu_N\left( \eta: \left| \dfrac{1}{N-1}\sum_{x\in I_N}H(x/N)\eta(x) - 
\int_0^1 H(u) \rho_0(u) du \right | \ge \delta\right)\;=\; 0 \;.
\end{align*}
Then, for any $t\ge0$, $\delta>0$ and continuous function $H:[0,1]\to\R$, we have
\begin{align*}
\lim_{N\to\infty}\mu_NS_t^N\left( \eta: \left| \dfrac{1}{N-1}\sum_{x\in I_N}H(x/N)\eta(x) - 
\int_0^1 H(u) \rho(t,u) du \right | \ge \delta\right)\;=\; 0 \;.
\end{align*}
where $\rho=\rho(\theta):[0,\infty)\times[0,1]\to[0,1]$ stands for
the unique weak solution to $\HDE$.
\end{thm}

\subsection{Hydrostatic limit}\label{2.3}

We examine in this subsection the hydrostatic limit established in Baldasso et al. \cite{bmns17} and outline their proof
to clarify the difference between their approach and ours.
The hydrodynamic limit describes the dynamical behavior of the empirical measure
while the hydrostatic limit states the law of large numbers for the empirical measure under 
the stationary state $\mu_N^{ss}$.

For each $\theta\ge0$, let $\rho_\theta:[0,1]\to[0,1]$ be the function defined by
\begin{align*}
\rho_\theta(u)\;=\;
\begin{cases}
     \rho_{D}(u)\;=\; (\b-\a)u+\a\;,  & \text{if $\theta<1$}\;,\\
     \rho_{R}(u)\;=\;\dfrac{c(\b-\a)}{2+c}u + \a + \dfrac{\b-\a}{2+c}\;,  & \text{if $\theta=1$}\;,\\
     \rho_{N}(u)\;=\;\dfrac{\b+\a}{2}\;,  & \text{if $\theta>1$}\;.\\
\end{cases}
\end{align*}
Note that, for each $\theta\ge0$, $\rho_\theta$ is a stationary solution to $\HDE$.

The following result has been established in Baldasso et al. \cite{bmns17}.

\begin{thm}[Hydrostatic limit]\label{11}
For any $\delta>0$ and continuous function $H:[0,1]\to\R$, we have
\begin{align*}
\lim_{N\to\infty}\mu_N^{ss}\left( \eta: \left| \dfrac{1}{N-1}\sum_{x\in I_N}H(x/N)\eta(x) - 
\int_0^1 H(u) \rho_\theta(u) du \right | \ge \delta\right)\;=\; 0 \;.
\end{align*}
\end{thm}

We here outline the proof given in Baldasso et al. \cite{bmns17}.
Their proof is summarized as follows.
For $x,y\in I_N$, let $\rho^N(x)$ and $\phi^N(x,y)$ be the mean of $\eta(x)$
and the two-point correlation function of $\eta(x), \eta(y)$ under the stationary state $\mu_N^{ss}$, respectively:
\begin{align*}
\rho^N(x)\;&=\; \int_{\Om_N} \eta(x) \mu_N^{ss}(d\eta)\;,\\
\phi^N(x,y)\;&=\; \int_{\Om_N} \left[ \eta(x) - \rho^N(x) \right] \left[ \eta(y) - \rho^N(y)\right]  \mu_N^{ss}(d\eta)\;.
\end{align*}

Since $\mu_N^{ss}$ is invariant with respect to $L_N$, for each $x\in I_N$, we have
\begin{align*}
\int_{\Om_N} L_N\eta(x) \mu_N^{ss}(d\eta)\;=\;0 \;.
\end{align*}
Computing the left-hand side, we can obtain a system of linear equations
for $\{\rho^N(x):x\in I_N\}$, see the proof of \cite[Lemma 3.1]{bmns17} for this system.
Since this system is linear, it is not difficult to obtain the explicit formula
\begin{align}\label{24}
\rho^N(x)\;=\; a_Nx+b_N\;, \quad x\in I_N\;,
\end{align}
where
\begin{align*}
a_N\;=\;\dfrac{c(\b-\a)}{2N^\theta +c(N-2)}\quad \text{and}\quad b_N\;=\; \a+a_N\left( \dfrac{N^\theta}{c}-1\right)\;.
\end{align*}

A similar computation for $\left[ \eta(x) -\rho^N(x) \right] \left[ \eta(y) - \rho^N(y) \right]$
together with some coupling argument
permits us to obtain the estimate
\begin{align}\label{25}
\max_{0<x<y<N}\left| \phi^N(x,y) \right| \;\le\;\dfrac{C}{N^\theta+N}\;,
\end{align}
for some constant $C>0$.
Theorem \ref{11} easily follows from \eqref{24}, \eqref{25}
and standard arguments based on the Chebyshev inequality.

We conclude this subsection mentioning a few comments on the proof.
For $\theta \le 1$, note that the function $\rho_\theta$ is the unique stationary solution to $\HDE$.
This fact together with Proposition \ref{13} below implies Theorem \ref{11} immediately.
However, for $\theta>1$, the set of stationary solutions to $\HDE$ is not a singleton,
since the corresponding Neumann Laplacian on $[0,1]$ has the eigenvalue $0$ in its spectrum.
Therefore, the concentration result, given in Proposition \ref{13}, does not imply Theorem \ref{11}
for $\theta>1$. To overcome this difficulty, besides Proposition \ref{13},
we need another characterization of the density $(\a+\b)/2$ amoung $[0,1]$,
which is a limiting density in the case $\theta>1$.
Indeed, we will see that $(\a+\b)/2$ can be characterized as a unique attractor of the integral equation \eqref{17}.
This is what we will investigate in Section \ref{sec4}.

An approach based on the estimates for the correlation functions is very useful
for several problems if available, see for instance \cite{fgn18, efgnt18, gjnm18}.
However, this approach can not be applied to almost all interacting systems,
even so-called {\it gradient} systems.
Compared to this approach, the method developed in this paper is robust
enough to be applicable to a gradient particle system
(should be possible for a {\it non-gradient} system).
For instance, one can obtain similar results for the setting of Farfan et al. \cite{flm11} with slow boundary.

\section{The diffusive time scale $N^2$}\label{sec3}
We investigate in this section the diffusive time scale,
to analyze the empirical measure under $\mu_N^{ss}$.
As examined in Section \ref{sec1},
we shall follow the method developed in Landim and Tsunoda \cite{lt18},
to prove some concentration result, stated in Proposition \ref{13}.
As a direct consequence of Proposition \ref{13},
which is a main result of this section, we shall prove Theorem \ref{11} for $\theta\le1$.

Let $\M_+$ be the set of all Borel measures on $[0,1]$, whose total mass is bounded above by $1$.
$\M_+$ is endowed with the weak topology, which is metrizable and becomes a compact Polish space.
Denote its metric by $d$, see for instance \cite[Subsection 2.2]{lt18} for the definition of $d$.
For a masure $\pi\in\M_+$ and a function $H:[0,1]\to\R$,
denote by $\lan \pi, H \ran $ the integral of $H$ with respect to $\pi$
whenever it has a meaning.
For functions $H_1,H_2:[0,1]\to\R$,  we also denote by $\lan H_1, H_2 \ran$ the $L^2$-inner product
with respect to the Lebesgue measure $du$ whenever it has a meaning.

For a configuration $\eta\in\Om_N$, define the empirical measure by
\begin{align*}
\pi_N(du) \;=\;\pi(\eta, du)\;=\; \dfrac{1}{N-1}\sum_{x\in I_N} \eta(x)\de_{x/N}(du)\;,
\end{align*}
where $\de_u$ stands for the point mass at $u\in[0,1]$.
Recall the definition of the stationary state $\mu_N^{ss}$, introduced at the last
paragraph of Subsection \ref{2.1}.
Define the probability measure $\cP_N$ on $\M_+$ by
\begin{align*}
\cP_N=\mu_N^{ss}\circ (\pi_N)^{-1}\;.
\end{align*}

For each $\theta\ge0$, let $\cE_\theta$ be the set of all measures $\pi(du)=\rho(u)du$ in $\M_+$ whose density
is a  stationary solution to $\HDE$. It is easy to see that $\cE_\theta$ coincides with
$\{\rho_D(u)du\}$ for $\theta<1$, $\{\rho_R(u)du\}$ for $\theta=1$ and 
$\{\varrho du: \varrho\in[0,1]\}$ for $\theta>1$, respectively.

Following the proof of \cite[Theorem 2.2]{lt18}, we can prove the following proposition:

\begin{prp}\label{13}
The sequence of measures $\{\mathcal P_N\}_{N\in\N}$ asymptotically concentrates on the set $\mathcal E_\theta$.
Namely, for any $\delta>0$, we have
\begin{align*}
\lim_{N\to\infty} \mathcal P_N\left( \pi\in\M_+: \inf_{\overline\pi\in\cE_\theta} d(\pi, \overline\pi)\ge\delta \right) \;=\; 0\;.
\end{align*}
\end{prp}
The proof of this proposition is consisting of two main ingredients, as examined in the first paragraph of \cite[Section 3]{lt18}:
the macroscopic density of the system is described by a hydrodynamic limit,
and for any initial profile the solution of the hydrodynamic equation converges to
some stationary solution as time goes to infinity.
Indeed, the exclusion process with slow boundary and its hydrodynamic equation satisfy
these two properties for any $\theta\ge0$.
Invoking these properties, the proof of Proposition \ref{13} is
completely same as the one of \cite[Theorem 2.2]{lt18} and thus is omitted.

Note that $\cE_\theta$ is a singleton for each $\theta\le1$: $\cE_\theta=\{\rho_D(u)du\}$ or $\{\rho_R(u)du\}$.
Theorem \ref{11} for $\theta\le1$ follows from Proposition \ref{13} immediately.

\begin{proof}[Proof of Theorem \ref{11} for $\theta\le 1$]
From Proposition \ref{13} and the fact that
$\cE_\theta$ is a singleton for each $\theta\le1$, the empirical measure $\pi_N$ under $\mu_N^{ss}$
converges to $\rho_\theta(u)du$ as $N\to\infty$ in probability.
Therefore, for any continuous function $H:[0,1]\to\R$, the random variable $\lan \pi_N, H\ran$ under $\mu_N^{ss}$
converges to $\lan \rho_\theta, H \ran$ as $N\to\infty$ in probability,
which completes the proof of Theorem \ref{11} for $\theta\le 1$.
\end{proof}

\section{The sub-diffusive time scale $N^{1+\theta}$}\label{sec4}
In the rest of this paper, we always treat with the case $\theta>1$.
Since the solution to the heat equation with $0$-Neumann boundary conditions conserves the total mass,
the total number of particles in $I_N$ can not evolve under the diffusive time scale.
This is exactly caused by the presence of slow boundary.
However, at the process level, we can observe exchange of particles
at a rate proportional to $N^{-\theta}$ through the boundary.
Therefore, to observe the correct evolution of the total number of particles in $I_N$,
we need to introduce another time scale, which should be longer than the diffusive time scale.
As understood in the computations below, the correct speeded up factor (or the time scale) is given by $N^{1+\theta}$,
which is in fact longer than the diffusive time scale in the case $\theta>1$.

For the sake of the previous paragraph, let $\{\eta_t^N:t\ge0\}$ be the Markov process generated by $N^{1+\theta}L_N$
with the initial distribution $\mu_N^{ss}$.
For each $t\ge0$, dente by $m_t^N$ the averaged density defined by
\begin{align*}
m_t^N\;=\;\dfrac{1}{N-1}\sum_{x\in I_N}\eta_t^N(x)\;.
\end{align*}
By the reason examined in the previous paragraph,
$m_t^N$ does not evolve under the diffusive time scale $N^2$.
On the other hand, as we will see later, $m_t^N$ evolves macroscopically under the time scale $N^{1+\theta}$.

For each $T>0$, let $D([0,T], \R)$ be the set of all c\`adl\`ag trajectories $m_\cdot:[0,T]\to\R$,
endowed with the Skorokhod topology.
For each $N\in\N$, let $Q_N=Q_{N,T}$ be the distribution of
$\{m_{\cdot}^N\}$ on $D([0,T], \R)$.

Our approach to study the sequence $\{Q_N\}_{N\in\N}$
is based on a standard machinery used in the study of hydrodynamic limit.
We first show the relative compactness of the sequence $\{Q_N\}_{N\in\N}$
and characterize its all limit points.
This is the content of Propositions \ref{12}, \ref{16} below, respectively.

We start with the relative compactness of the sequence $\{Q_N\}_{N\in\N}$.

\begin{prp}\label{12}
The sequence $\{m_\cdot^N\}_{N\in\N}$ is relatively compact in $D([0,T], \R)$.
\end{prp}
\begin{proof}
Fix $T>0$. It is enough to show that
the sequence $\{m_\cdot^N\}_{N\in\N}$ is relatively compact in $D([0, T], \R)$.
For this purpose, introduce the function $G_N(\eta)=G(\eta)=(N-1)^{-1}\sum_{x\in I_N} \eta(x)$ and the corresponding Dynkin's martingale:
\begin{align}\label{02}
M_t^N\;&=\; G(\eta_t^N) -G(\eta_0^N) -N^{1+\theta}\int_0^t (L_NG) (\eta_s^N) ds\;, \quad t\ge0\;.
\end{align}

It follows from the definition of $m_t^N$ that $G(\eta_t^N)=m_t^N$.
Since the total number of particles in $I_N$ is conserved by $L_{N,0}$, we have $L_{N,0}G = 0$.
One the other hand, as $L_{N,b}^\a$ and $L_{N,b}^\b$ act only at the left and right boundary, respectively, we have
\begin{align*}
(L_{N,b}^\a+L_{N,b}^\b)G(\eta)
\;&=\; \dfrac{c}{N^{\theta}(N-1)}\left\{ r_\a(\eta)\left[1-2\eta(1)\right] + r_\b(\eta)\left[1-2\eta({N-1})\right] \right\}\\
\;&=\; \dfrac{c}{N^{\theta}(N-1)}\left[\a+\b-\eta(1)-\eta(N-1)\right]\;.
\end{align*}
Therefore \eqref{02} can be rewritten as
\begin{align}\label{14}
m_t^N\;&=\; m_0^N+ M_t^N + \dfrac{cN}{N-1}\int_0^t \left[\a+\b-\eta_s^N(1)-\eta_s^N(N-1)\right] ds\;\\\notag
\;&=\; m_0^N+ M_t^N + c\int_0^t \left[\a+\b-\eta_s^N(1)-\eta_s^N(N-1)\right] ds + O(N^{-1})\;,
\end{align}
where big $O$ notation stands for the Bachman-Landu notation.

Note that the sequence $\{ m_0^N\}_{N\in\N}$ is relatively compact since $m_0^N$ takes values in $[0,1]$ for any $N\in\N$.
On the other hand, in view of Aldous's criterion, cf. \cite[page 51, Proposition 4.1.6]{kl99}, we can obtain the relative
compactness of the integral term in the last line of \eqref{14}.
Therefore, to conclude the proof, it is enough to show that the sequence $\{ M_\cdot^N\}_{N\in\N}$
is relatively compact in $D([0,T], \R)$.

Indeed, it follows from a straightforward computation that the quadratic variation of $M_t^N$
is given by
\begin{align}\label{30}
\dfrac{cN}{(N-1)^2}&\int_0^t
\left| \eta_s^N(1)-\a \right| + \left| \eta_s^N(N-1)-\b \right| ds
\;=\; O(N^{-1})\;.
\end{align}
This formula together with the standard argument as in the proof of \cite[page 55, Theorem 4.2.1]{kl99}
gives the relative compactness for the sequence $\{ M_\cdot^N\}_{N\in\N}$,
which completes the proof of Proposition \ref{12}.
\end{proof}

It follows from \eqref{30} that the martingale term $M_t^N$ vanishes in the limit $N\to\infty$.
Therefore, if we can replace $\eta_s^N(1)+\eta_s^N(N-1)$ by $2m_s^N$ in \eqref{14},
we can obtain the following integral equation in the limit:
\begin{align}\label{17}
m_t\;=\; m_0 + c\int_0^t \left[\a+\b -2m_s \right] ds\;, \quad t\ge0\;.
\end{align}
This replacement can not be achieved in the diffusive time scale
since the relaxation time, which is the inverse of the spectral gap,
of the exclusion process inside a box with side length $\ell$ is of order $\ell^2$.
However, such a replacement should be achieved in the time scale $N^{1+\theta}$.
This is the idea hidden in the proof of Lemma \ref{04}, so-called {\it replacement lemma}.

Before starting the proof of the replacement lemma,
we introduce some notation and estimates, which will be used in the proof of the replacement lemma.

For two probability measures $\mu, \nu$ on $\Om_N$, let $H_N(\mu|\nu)$
be the {\it relative entropy} of $\mu$ with respect to $\nu$:
\begin{align*}
H_N(\mu|\nu)\;=\; \sup_f \left\{\int_{\Om_N} fd\mu -\log \int_{\Om_N} e^f d\nu\right\} \;,
\end{align*}
where the supremum is carried over all functions $f:\Om_N\to\R$.
It is well known that
\begin{align*}
H_N(\mu|\nu)\;=\;
\int_{\Om_N} \dfrac{d\mu}{d\nu}\, \log{\dfrac{d\mu}{d\nu}} \, d\nu \;,
\end{align*}
if $\mu$ is absolutely continuous with respect to $\nu$,
$H_N(\mu|\nu) = \infty$, otherwise.
Since there is  at most one particle per site,
there exists a constant $C_0=C_0(\a)>0$ such that
\begin{align}\label{20}
H_N(\mu|\nu_\a^N) \;\le\; C_0N\;,
\end{align}
for any probability measure $\mu$ on $\Om_N$.

A function $f:\Om_N\to[0,\infty)$ is said to be a {\it density} if $\int f d\nu_\a^N=1$.
For any density $f$, define the {\it Dirichlet form} with respect to $\nu_\a^N$ by
\begin{align*}
D_{N,0}(f;\nu_a^N)\;=\; \dfrac{1}{2}\sum_{x=1}^{N-2}\int_{\Om_N}
\left[ \sqrt{f(\eta^{x,x+1})} - \sqrt{f(\eta)} \right]^2 d\nu_\a^N\;.
\end{align*}
It has been established in the proof of \cite[Lemma 5.9]{bmns17} that
there exists a constant $C_{\a,\b}>0$ such that
\begin{align}\label{21}
\lan L_N\sqrt f, \sqrt f\ran_\a\;\le\; -D_{N,0}(f;\nu_\a^N)+ \dfrac{C_{\a,\b}}{N^\theta}\;,
\end{align}
for any density $f$, where $\lan \cdot, \cdot \ran_\a$ stands for the $L^2$-inner product with respect to $\nu_\a^N$.
Since the actual value of the density of the reference measure is not important, we always fix it to be $\a$.

From the observation examined after the proof of Proposition \ref{12},
introduce the function $V=V_N$ given by $V(\eta)=\eta(1)+\eta(N-1) - 2G(\eta)$,
where $G$ has been introduced in the proof of Proposition \ref{12}.

We are ready to prove the replacement lemma.

\begin{lmm}[Replacement lemma]\label{04}
For any $t\ge0$, we have
\begin{align}\label{03}
\lim_{N\to\infty} \E^N \left[  \left | \int_0^t V(\eta_s^N) ds  \right| \right] \;=\; 0 \;,
\end{align}
where $\E^N$ stands for expectation with respect to the process $\eta_\cdot^N$.
\end{lmm}
\begin{proof}
For any $\ga>0$, from the entropy inequality and \eqref{20}, we have
\begin{align}\label{05}
 \E^N \left[  \left | \int_0^t V(\eta_s^N) ds  \right| \right] \;\le\;
 \dfrac{C_0}{\ga}+\dfrac{1}{\ga N}\log \E_\a \left[ \exp \left \{ \ga N\left| \int_0^t V(\eta_s^N) ds \right| \right\}\right]\;,
\end{align}
where $\E_\a$ stands for expectation with respect to the process starting from the product measure $\nu_\a^N$.
One can get rid of the absolute value in the right-hand side of \eqref{05} by the elementary inequality
$e^{|x|}\le e^x +e^{-x}$. 
Therefore, the estimate for the second term in the right-hand side of \eqref{05}
is reduced to the one without the absolute value.
Furthermore, from \cite[Lemma 7.3]{bmns17}, to conclude the proof, it is enough to
show that the following variational expression vanishes as $N\to\infty$ and $\ga\to\infty$:
\begin{align}\label{22}
\sup_{f}\left \{\ga^{-1}N^\theta\lan L_N\sqrt f, \sqrt f\ran_\a +\lan V, f\ran_\a \right\}\;,
\end{align}
where the supremum is carried over all densities $f$.

From \eqref{21} and Lemma \ref{06} below, the supremum \eqref{22} is bounded above by
\begin{align}\label{07}
\dfrac{C_{\a,\b}}{\ga}+\sup_f\left \{-\ga^{-1}N^\theta D_{N,0}(f;\nu_\a^N) + 4N^{1/2}D_{N,0}(f;\nu_\a^N)^{1/2}\right \}\;.
\end{align}
The previous supremum is easily computed and is bounded above by $4\ga N^{1-\theta}$.
Since $\theta$ is larger than $1$, the expression \eqref{07} vanishes as $N\to\infty$ and $\ga\to\infty$,
which completes the proof of Lemma \ref{04}.
\end{proof}

The following lemma in fact has been proved in the proof of \cite[Lemma 3.1]{lan92}.
However, we give the proof for reader's convenience.

\begin{lmm}[Moving particle lemma]\label{06}
For any density $f$, we have
\begin{align}\label{08}
\lan V, f \ran_\a\;\le\; 4N^{1/2}D_{N,0}(f;\nu_\a^N)^{1/2}\;.
\end{align}
\end{lmm}
\begin{proof}
Fix a density $f$. The left-hand side in \eqref{08} can be written as
\begin{align}\label{09}
\dfrac{1}{N-1}\sum_{x\in I_N} \int_{\Om_N}\left\{ \left[\eta(1)-\eta(x)\right ] + \left[ \eta(N-1) -\eta(x) \right] \right\}f(\eta)\nu_\a^N(d\eta)\;.
\end{align}
In the following argument, we give an estimate for the sum involving $\eta(1)-\eta(x)$ only since the other sum is similar.

For each $x\in I_N$, by the change of variables $\eta\mapsto\eta^{1,x}$, 
the sum involving $\eta(1)-\eta(x)$ in \eqref{09} can be rewritten as
\begin{align*}
&\dfrac{1}{2(N-1)}\sum_{x\in I_N} \int_{\Om_N} \left[\eta(1)-\eta(x) \right]
\left[ f(\eta) - f(\eta^{1,x}) \right]\nu_\a^N(d\eta)\\
&= \dfrac{1}{2(N-1)}\sum_{x\in I_N} \int_{\Om_N} \left[\eta(1)-\eta(x) \right]
\left[\sqrt{f(\eta)} + \sqrt{f(\eta^{1,x})}\right] \left[\sqrt{f(\eta)} - \sqrt{f(\eta^{1,x})}\right]  \nu_\a^N(d\eta)\;.
\end{align*}
Since there is at most one particle per each site and $f$ is a density,
applying the Schwartz inequality, the last expression is bounded above by
\begin{align}\label{10}
\dfrac{1}{(N-1)}\sum_{x\in I_N}\left\{  \int_{\Om_N} 
\left[\sqrt{f(\eta)} - \sqrt{f(\eta^{1,x})}\right]^2  \nu_\a^N(d\eta) \right\}^{1/2}\;.
\end{align}

For each $x,y\in I_N$, consider the transformation $\sigma^{x,y}$ on $\Om_N$ defined by $\sigma^{x,y}\eta=\eta^{x,y}, \eta\in\Om_N$.
Clearly, $\sigma^{x,y}$ is $\nu_\a^N$-measure preserving and satisfies the relation
\begin{align*}
\sigma^{1,x}\;=\; \sigma^{1,2}\circ\sigma^{2,3}\circ\cdots\circ\sigma^{x-2,x-1}\circ\sigma^{x,x-1}\circ\sigma^{x-1,x-2}
\circ\cdots\circ\sigma^{3,2}\circ\sigma^{2,1}\;,
\end{align*}
for any $x\in I_N$, where the symbol $\circ$ stands for the composition of transformations.
By adding and subtracting the terms by this sequence into the brackets in \eqref{10}, from the Cauchy-Schwarz inequality,
we have
\begin{align*}
 \int_{\Om_N} \left[\sqrt{f(\eta)} - \sqrt{f(\eta^{1,x})}\right]^2  \nu_\a^N(d\eta)
 \;\le\; 4ND_{N,0}(f;\nu_\a^N)\;,
\end{align*}
which in turn implies the conclusion of Lemma \ref{06}. Note that the constant $4$ in \eqref{08}
comes from the contribution of the sum involving  $\eta(N-1)-\eta(x)$ in \eqref{09}. 
\end{proof}

We summarize the previous  computations as a single proposition,
which plays a fundamental role in the proof of Theorem \ref{11}.
However, as the proof follows from the formula \eqref{14} and Lemma \ref{04} easily, we omit the proof.

\begin{prp}\label{16}
Let $\A$ be the set of all trajectories $\{m_t: t\ge0\}$ in $D([0,T], \R)$ satisfying the integral equation \eqref{17}
with the initial value $m_0$ in $[0,1]$.
Then, any limit point $Q_*$ of the sequence $\{Q_N\}_{N\in\N}$ is concentrated on $\A$,
namely, $Q_* \left( \A \right)=1.$
\end{prp}

The following lemma states that $(\a+\b)/2$ can be characterized
as a unique attractor of the integral equation \eqref{17}.

\begin{prp}\label{26}
The solution of the integral equation \eqref{17} is given by
\begin{align*}
m_t\;=\; \dfrac{\a+\b}{2} +\left(m_0 - \dfrac{\a+\b}{2} \right) e^{-2ct}\;.
\end{align*}
In particular, $m_t$ converges to $(\a+\b)/2$ as $t\to\infty$, uniformly in initial values in $[0,1]$.
\end{prp}
The proof of this proposition is elementary, and left to the reader.

We have now all the ingredients to prove Theorem \ref{11} for $\theta>1$.

\begin{proof}[Proof of Theorem \ref{11} for $\theta>1$]
Since $\M_+$ is compact, the sequence $\{\cP_N\}_{N\in\N}$ is relatively compact.
Let $\cP_*, Q_*$ be any limit point of the sequence $\{\cP_N\}_{N\in\N}, \{Q_N\}_{N\in\N}$, respectively.
Take a subsequence $N_k$, if necessary, so that the sequences $\{\cP_{N_k}\}_{k\in\N}, \{Q_{N_k}\}_{k\in\N}$
converge to $\cP_*, Q_*$, respectively.
Note that Proposition \ref{13} shows that $\cP_*(\cE_\theta)=1$.
Recall the definition of the function $\rho_N$: $\rho_N(u)=(\a+\b)/2, u\in[0,1]$.
To conclude the proof, it is enough to show that
\begin{align}\label{23}
\cP_*\left(\{\rho_N(u)du\}\right)\;=\;1\;.
\end{align}

Fix $\delta>0$. Denote by $O_{\delta}$ the subset of $[0,1]$ given by
\begin{align*}
O_{\delta}\;=\; \left[0,1\right]\setminus \left[\dfrac{\a+\b}{2}-\de, \dfrac{\a+\b}{2}+\de\right]\;,
\end{align*}
and by $\overline{O_\delta}$ the closure of $O_\delta$.
From the stationarity of $\mu_N^{ss}$, we have
\begin{align}\label{15}
\cP_N \left(\pi: \lan \pi, {\bf 1}\ran\in O_\delta\right)\;=\;\mu_N^{ss}\left(\eta: G(\eta)\in O_\delta\right)
\;=\; Q_N\left(m_\cdot: m_t\in O_\delta\right)\;,
\end{align}
for any $t\ge0$, where ${\bf 1}$ stands for the constant function ${\bf 1}(u)=1, u\in[0,1]$.

Since the application $\pi\mapsto\lan\pi, {\bf 1}\ran$ is continuous with respect to the weak topology,
and $\cP_{N_k}, Q_{N_k}$ converge to $\cP_*, Q_*$ weakly, respectively, we have
\begin{align*}
\cP_*\left(\pi: \lan\pi, {\bf 1}\ran \in O_{2\delta}\right) \;&\le\; \liminf_{k\to\infty}\cP_{N_k}\left(\pi: \lan\pi, {\bf 1}\ran \in O_{2\delta}\right)\\
\;&=\; \liminf_{k\to\infty}Q_{N_k}\left(m_\cdot: m_t \in O_{2\delta}\right)\\
\;&\le\; \limsup_{k\to\infty}Q_{N_k}\left(m_\cdot: m_t\in \overline{O_\delta}\right)\\
\;&\le\; Q_*\left(m_\cdot: m_t\in \overline{O_\delta}\right)\;.
\end{align*}
We used \eqref{15} to obtain the second equality and the monotonicity of $Q_{N_k}$ the third inequality.
For the last inequality, one should pay an attention since the application $m_{\cdot}\mapsto m_t$ is not continuous
with respect to the Skorokhod topology. However, one can justify this inequality by the fact that $Q_*$
is concentrated on continuous trajectories, see the proof of \cite[Theorem 2.2]{lt18} for a similar argument.
Since $\overline{O_\delta}$ does not contain $(\a+\b)/2$,
it follows from Propositions \ref{16}, \ref{26} that $Q_*\left(m_\cdot: m_t\in \overline{O_\delta}\right)$ vanishes
if $t$ is larger than $-(2c)^{-1}\log\de$.
Thus, \eqref{23} has been shown. This completes the proof of Theorem \ref{11} for $\theta>1$.
\end{proof}

\section*{Acknowledgements}
The author would like to thank Professor Tomoyuki Shirai for 
giving him an opportunity to attend the international conference
{\it Stochastic dynamics out of equilibrium} held in Institut Henri Poincar\'e from June 12--16, 2017.
Most of the proof presented in Section \ref{sec4} has been given during this visit.
He is also grateful to the anonymous referee for the comments
that have been helpful in revising the paper.


\begin{thebibliography}{99}

\bibitem{bmns17} R. Baldasso, O. Menezes, A. Neumann and R. R. Souza: 
Exclusion process with slow boundary.
J. Stat. Phys. {\bf 167}, 1112--1142 (2017).

\bibitem{bdgjl15} L. Bertini, A. De Sole, D. Gabrielli, G. Jona-Lasinio and C. Landim:
Macroscopic fluctuation theory.
Rev. Modern Phys. {\bf 87}, 593--636 (2015).

\bibitem{flm11} J. Farfan, C. Landim and M. Mourragui: 
Hydrostatics and dynamical large deviations of boundary gradient symmetric exclusion processes.
Stochastic Process. Appl. {\bf 121}, 725--758 (2011).

\bibitem{efgnt18} D. Erhard, T. Franco, P. Gon\c{c}alves, A. Neumann and M. Tavares: 
Non-equilibrium fluctuations for the SSEP with a slow bond.
To appear in Ann. Inst. H. Poincar\'e Probab. Statist.

\bibitem{fgn18} T. Franco, P. Gon\c{c}alves  and A. Neumann: 
Non-equilibrium and stationary fluctuations of a slowed boundary symmetric exclusion.
online first, Stochastic Process. Appl. (2018).

\bibitem{fn17} T. Franco and A. Neumann: 
Large deviations for the exclusion process with a slow bond.
Ann. Appl. Probab. {\bf 27}, 3547--3587 (2017).

\bibitem{gjnm18} P. Gon\c{c}alves, M. Jara, A. Neumann and O. Menezes: 
Non-equilibrium and stationary fluctuations for the SSEP with slow boundary.
http://arxiv.org/abs/1810.05015, 2018.

\bibitem{kl99} C. Kipnis, C. Landim: {\it Scaling Limits of Interacting
    Particle Systems}. Grundlheren der mathematischen Wissenschaften
  {\bf 320}, Springer-Verlag, Berlin, New York, 1999.
  
\bibitem{lan92} C. Landim: 
Occupation time large deviations for the symmetric simple exclusion process.
Ann. Probab. {\bf 20}, 206--231 (1992).

\bibitem{lt18} C. Landim, K. Tsunoda: 
Hydrostatics and dynamical large deviations for a reaction-diffusion model.
Ann. Inst. Henri Poincar\'e Probab. Stat. {\bf 54}, 51--74 (2018).

\end{thebibliography}
\end{document}